\documentclass[12pt, a4paper]{article}
\usepackage{setspace}
\usepackage{geometry}
\usepackage[OT1]{fontenc}
\usepackage{amsthm,amsmath}
\usepackage[colorlinks,citecolor=blue,urlcolor=blue]{hyperref}
\usepackage{amssymb,amsmath,bm,mathrsfs,makeidx,amsfonts,graphicx,amsthm,multirow}
\usepackage{booktabs}
\newcommand\x{\mathbf x}
\newcommand\z{\mathbf z}

\renewcommand\Im{{\rm\,Im} }

\newcommand\bR{{\mathbb R}}
\newcommand\bZ{{\mathbb Z}}

\newcommand\wh{\widehat }

\newcommand\cov{\mathrm{Cov}}
\newcommand\bN{{\mathbb N}}
\newcommand\pto{\stackrel {\mathbb P}\rightarrow}
\newcommand\bI{\mathrm I}
\newcommand\e{{\mathbb E}}
\newcommand\p{{\mathbb P}}
\newcommand\var{{\rm Var}}
\newcommand\tr{{\mathrm{tr}}}

\newcommand\cB{{\mathcal B}}
\newcommand\bC{{\mathbb C}}

\newtheorem{theorem}{Theorem}[section]%
\newtheorem{lemma}[theorem]{Lemma}%
\newtheorem{example}[theorem]{Example}%
\newtheorem{definition}[theorem]{Definition}%

\begin{document}

\begin{center}
\Large Limiting spectral distribution for large sample covariance matrices with graph-dependent elements.
\end{center}
\begin{center}
\large Pavel~Yaskov\footnote{Steklov Mathematical Institute of RAS, Moscow, Russia\\
 e-mail: yaskov@mi-ras.ru\\This work is supported by the Russian Science Foundation under grant 18-71-10097.}
 \end{center}

\begin{abstract}  
We obtain the limiting spectral distribution for large sample covariance matrices associated with random vectors having graph-dependent entries under the assumption that the interdependence among the entries grows with the sample size $n$. Our results are tight. In particular, they give necessary and sufficient conditions for the Marchenko-Pastur theorem for sample covariance matrices with $m$-dependent orthonormal elements when $m=o(n)$.
\end{abstract}

\begin{center}
\end{center}

\section{Introduction}
The asymptotic behaviour of the spectrum of large sample covariance matrices plays an important role in high-dimensional statistical problems, in particular, those related to the least squares estimation and the covariance matrix estimation (e.g., see the papers \cite{Pr}, \cite{DW}, \cite{K}, and the references therein). The present paper studies  weak limits of the empirical spectral distributions of sample covariance matrices 
\begin{equation}
\label{e8}\wh \Sigma_n=\frac{1}{n}\sum_{k=1}^n \x_{pk}\x_{pk}^\top
\end{equation} 
under the assumption that $p,n\to\infty$ and $p/n\to\rho>0,$ where $\{\x_{pk}\}_{k=1}^n$ are i.i.d. copies of a random vector $\x_p$ in $\bR^p$ and the empirical spectral distribution of a symmetric matrix $A\in\bR^{p\times p}$ with eigenvalues $\lambda_1\leqslant\lambda_2\leqslant\ldots\leqslant \lambda_p$ is defined by 
\[\mu_{A}= \frac 1p\sum_{i=1}^p \delta_{\lambda_i}\]
with $\delta_\lambda$ being a Dirac measure with mass at $\lambda\in\bR$. Notice that in this construction, the sample mean vector is not subtracted from $\x_{pk}$, since it does not affect the limiting spectral distributions  (see the rank inequality of Theorem A.43 in \cite{BS}). 

There is a number of results in random matrix theory, allowing to compute the limiting spectral distribution of $\wh \Sigma_n$ under different distributional assumptions on $\x_{p}$. Let us mention the papers \cite{A}, \cite{BZ},  \cite{mp}, \cite{BVH}, \cite{GNT},  \cite{L}, \cite{PM}, \cite{MPP}, \cite{OR}, \cite{Y}, among others. The most general conditions imposed on $\x_{p}$ ensure that the quadratic forms $\x_{p}^\top  A_p\x_{p}$ weakly concentrate around their expectations up to an error term $o( p)$ with probability $1-o(1)$, where $A_p\in \bC^{p\times p}$ is an arbitrary matrix with the spectral norm $\|A_p\|\leqslant 1$. These conditions were studied in  \cite{BZ}, \cite{G}, \cite{L}, \cite{PP}, \cite{Y15}, \cite{Y16}, and \cite{Y18}. As shown in \cite{Y16}, the weak concentration property for specific quadratic forms of $\x_p$ gives necessary and sufficient conditions for the  Marchenko-Pastur theorem \cite{MP}. 

In general, the weak concentration property for the quadratic forms could hard to verify in practice. This presents a separate problem to be solved for a given data model. In this paper, we obtain concentration inequalities for $\x_{p}^\top  A_p\x_{p}$ with $\x_p$ having graph-dependent entries under the assumption that the interdependence among the entries grows with $p$. As a corollary, we derive the limiting spectral distribution for the large sample covariance matrices $\wh\Sigma_n$ associated with $\x_p$. Our results are close to that of \cite{BVH}, \cite{HP}, and \cite{WYY}. The paper \cite{BVH} studies the block independent model for $\x_p$ where the entries of $\x_p$ are partitioned into blocks in such a way that the entries in different blocks are independent and the blocks may grow with $p$. Following \cite{HP}, the paper \cite{WYY} considers the $m$-dependent model for $\x_p$, where the entries of $\x_p$ are $m$-dependent with $m$ growing with $p$. In contrast to these papers, our results are more general and tight. In particular, they give necessary and sufficient conditions in the isotropic case with $\e\x_p\x_p^\top$ being the identity matrix (for details, see Section \ref{mr}). 

Also, despite the fact that graph-dependent data are frequently appears in computer science problems, to the best of our knowledge, there are no well-established methods for proving concentration inequalities for quadratic forms in graph-dependent random variables with growing interdependence among them. This is in contrast to linear functions, which could be efficiently analysed via the method based on fractional coloring of the dependency graph \cite{J}, or bounded-difference functions, which could be analysed via the method relying on the forest complexity of the dependency graph \cite{ZWW}. In this paper, we use a straightforward approach for analysing the quadratic forms, which is based on an appropriate covering of the dependency graph by balls centered at vertices from its dominating set.   

The paper is structured as follows. Section \ref{mr} contains our main results. Section \ref{proofs}
deals with the proofs. Some additional results are given in an Appendix.

\section{Main results}\label{mr}

Let us introduce some notation.  Set $[\![a,b]\!]=[a,b]\cap \bZ$ for all $a\leqslant b$ and denote by $|S|$ a cardinality of a set $S$. For all $p\geqslant 1$, let $\x_p$ be a random vector in $\bR^p$ and let $\Sigma_p\in \bR^{p\times p}$ be symmetric positive semidefinite, hereinafter $\bR^{p\times p}$ stands for the set of all real  $p\times p$ matrices. 
For $A\in \bR^{p\times p}$, $\|A\|$ will denote its spectral norm. Also, set $\bC_+:=\{z\in\bC:\Im(z)>0\}$ and denote by $\cB(\bR_+)$ the Borel $\sigma$-algebra of $\bR_+$. All random elements will be defined on the same probability space.

First, let us recall some known results on the limiting spectral distribution of $\wh\Sigma_n$ from \eqref{e8}. They will be stated under the following general assumptions:

 (A1)  $(\x_p^\top A_p \x_p -\tr(\Sigma_p A_p ))/p\pto 0$ as $p\to\infty$ for all sequences of  symmetric positive semidefinite $A_p\in\bR^{p\times p}$ with $\|A_p\|\leqslant 1$.

 (A2)  $\tr(\Sigma_p^2)/p^2\to0 $ as $p\to\infty$.

\begin{theorem}\label{rm}
Let $p=p(n)\in\bN$ be such that  $p/n\to \rho>0$ when $n\to\infty$. If $\{(\x_p,\Sigma_p)\}_{p=p(n),n\geqslant 1}$ satisfies {\rm  (A1)}--{\rm  (A2)} and  $\mu_{\Sigma_{p }}$ converges weakly to a probability measure $\mu$  on $\cB(\bR_+)$, 
 then \[\p(\mu_{\wh \Sigma_n}\to  \nu \text{ weakly, } n\to\infty)=1,\] where $\nu$ is  a probability measure on $\cB(\bR_+)$, whose Stieltjes transform 
\[s(z)=\int_{\bR_+} \frac{\nu(d\lambda)}{\lambda-z},\;z\in\bC_+,\;\text{satisfies}\;s(z)=\int_{\bR_+} \frac{\mu(d\lambda)}{\lambda(1-\rho-\rho zs(z))-z}.\]
\end{theorem}
As is shown in \cite{Y18}, the above theorem  follows from Theorem 2 in \cite{Y18} and Theorem 7.2.2 with Remark 7.2.6.(4) in \cite{PS}.
In the case with i.i.d. $\x_{pk}$, Theorem \ref{rm}  extends Theorem 1.1 in \cite{BZ} by allowing $\|\Sigma_p\|$ to be unbounded and replacing convergence in $L_2$ by convergence in probability in (A1). The general case with independent $\x_{pk}$ could considered  similarly (e.g., see Remark 1 in \cite{Y14}).   

When the weak limit of $\mu_{\Sigma_p}$ is $\delta_1$, the measure $\nu$ in Theorem \ref{rm} is just the Marchenko-Pastur law $\mu_\rho$ with parameter $\rho>0$, which is defined by 
\[d\mu_\rho=\max\{1-1/\rho,0\}\,d\delta_0+\frac{\sqrt{(b-x)(x-a)}}{2\pi x\rho   }I(x\in [a,b])\,dx, \]
where  $a=(1-\sqrt{\rho})^2 $ and $b=(1+\sqrt{\rho})^2.$ 
We can state a stronger result in the isotropic case with $\e \x_p\x_p^\top=I_p$.
\begin{theorem}\label{mp1}
Let $p=p(n)\in\bN$ satisfy  $p/n\to \rho>0$ as $n\to\infty$. 
If, for $p=p(n)$, $\x_p$ is a random vector in $\bR^p$ with $\e \x_p\x_p^\top=I_p$  and {\rm (A1)} holds for $\Sigma_p=I_p$, then
\begin{equation}\label{mpp}
\p(\mu_{\wh \Sigma_n}\to  \mu_\rho\text{ weakly, } n\to\infty)=1.
\end{equation}
Furthermore, if \eqref{mpp} holds, then $\x_p^\top\x_p/p\pto 1$ as $p=p(n)\to\infty$.
\end{theorem}
 The sufficient part of Theorem \ref{mp1} follows from Theorem \ref{rm}, the necessity part follows from Theorem 2.1 in \cite{Y16}. As we will see below, the necessary condition $\x_p^\top\x_p/p\pto 1$ or its anisotropic analogue  will be also a sufficient condition in the following graph-dependent model.

\begin{definition}\label{gd} \normalfont Consider a random vector $\x=(X_k)_{k=1}^p$ in $\bR^p$. We say that $\x$ follows {\it the graph dependent model} with an undirected graph $G=(V,E)$ having the vertex set $V=[\![1,p]\!]$ and an edge set $E$ such that the collections 
$\{X_i\}_{i\in I}$ and $\{X_j\}_{j\in J}$ are independent when $I,J\subseteq V$ are non-adjacent.\footnote{Here $I,J\subseteq V$  are adjacent in $G$ if there are $i\in I,j\in J$ such $i=j$ or $i$ is a neighbour of $j$ in $G$ and  $I,J$ are non-adjacent in $G$ otherwise.}
\end{definition}

To state our main results, we need to introduce some more definitions for a graph $G=(V,E)$. Let $d_G$ be the distance in $G$, i.e. $d_G(u,v)$ is the number of edges in a shortest path connecting $u,v\in V$, $d_G(u,v)=0$ if $u=v,$ and $d_G(u,v)=\infty$ if there is no path connecting $u,v$ and $u\neq v$. In what follows, we also set $B_d(v)=B_d(v;G):=\{u\in G: d_G(u,v)\leqslant d\}$. For $d\in\bZ_+,$ we will say that $\mathcal V\subseteq V$ is a $d$-dominating set for $G$ if every vertex not in $\mathcal V$ is adjacent to at least one vertex in $\mathcal V$ and, for any $u\in V$, there are no more than $d$ vertices $v\in\mathcal V$ with $d_G(u,v)\leqslant 3$. Recall also that the maximum degree of a graph is the maximum of its vertices' degrees.

We can now state our main concentration inequality for quadratic forms of graph-dependent random variables. 
\begin{theorem}\label{t3} Let $\x=(X_k)_{k=1}^p$ be a random vector in $\bR^p$ with mean zero and covariance matrix  $\Sigma$. Suppose $\x$ follows the graph dependent model with a graph $G$. If $G$ has the maximum degree $\Delta\geqslant 0$ and a $d$-dominating set for some $d\in\bN$, then, for all symmetric $A=(a_{ij})_{i,j=1}^p\in\bR^{p\times p}$, 
\begin{equation}\label{t3-1}
\var(\x^\top A\,\x)\leqslant  C_d \|A\|^2  (\Delta+1) \sum_{k=1}^p\e X_k^4,
\end{equation}where $C_d>0$ depends only on $d$. If, in addition, $a_{ij}=0$ when $d_G(i,j)\leqslant 2$, then
\begin{equation}\label{t3-2} \var(\x^\top A\,\x)\leqslant 2\|A\|^2\tr(\Sigma^2).
\end{equation}
\end{theorem}

The proof of Theorem \ref{t3} is deferred to Section \ref{proofs}.  Inspecting the proof shows that one could take $C_d=(d^7+2)2^{2d}$. Furthermore, the bounds \eqref{t3-1}--\eqref{t3-2} will hold even if we correct  Definition \ref{gd} by assuming that the entries of $\x_p$ have finite fourth moments, the edge set $E$ is such that the covariances $\cov(X_iX_j,X_kX_l)$, $\cov(X_i,X_k)$, $\cov(X_i,X_l)$, $\cov(X_j,X_k)$, $\cov(X_j,X_l)$ are zero for all non-adjacent sets $\{i,j\},\{k,l\}\subseteq V$ (here we allow the cases $i=j$ and $k=l$). 

Theorem \ref{t3}  allows to verify  (A1) for $\Sigma_p=\e \x_p\x_p^\top$ in different scenarios, where  $\x_p$ follows a graph dependent model for each $p\geqslant 1$ and the model parameters $d=d(p)=O(1)$ and $\Delta=\Delta(p)=o(p)$ as $p\to\infty$.

\begin{theorem}\label{mdep}
Let $p=p(n)\in\bN$ be such that  $p/n\to \rho>0$, hereinafter all limits are with respect to $n\to\infty$. Assume also that

{\rm (i)} for each $p=p(n)$, a zero-mean random vector $\x_p$ in $\bR^p$ follows the graph dependent model with a graph $G_p$ having the maximum degree $\Delta_p=o(p)$ and a $d_p$-dominating set with $d_p=O(1)$,

{\rm(ii)} there exist $c> 0$ and a probability measure $\mu$ on $\cB(\bR_+)$ such that $\mu_{\Sigma_p}\to \mu$ weakly and $\tr(\Sigma_p)/p\to c$, where $\Sigma_p=\e \x_p\x_p^\top$.
\\
If $\x_p^\top\x_p/p\pto c$, then $\p(\mu_{\wh \Sigma_n}\to  \nu \text{ weakly})=1,$ where $\nu$ is defined in Theorem \ref{rm}. Furthermore, if $\e\x_p\x_p^\top =I_p$ for each $p=p(n)$, then  
\[\x_p^\top\x_p/p\pto 1\quad\text{iff}\quad\p(\mu_{\wh \Sigma_n}\to  \mu_\rho \text{ weakly})=1,\] where $\mu_\rho$ is the Marchenko-Pastur law with parameter $\rho$.
\end{theorem}
The proof of the above theorem is deferred to Section \ref{proofs} and based on Theorem \ref{rm} and \ref{mp1}. 
Let us consider two examples of graph dependent models.

\begin{example}\label{bi}\normalfont (Block-independent model)  Suppose the entries of $\x_p=(X_{pk})_{k=1}^p$ can be partitioned into $q$ blocks $ I_{pk}$ each of length $d_{pk}$ ($k\in[\![1,q]\!]$), in such
a way that the entries in different blocks are independent. Assume also that  the entries have zero mean, unit variance, and fourth moment bounded by $K>0$. Then $\x_p$ follows the graph dependent model with $G_p=(V_p,E_p)$, where $V=[\![1,p]\!]$ and \[E_p=\{\{i,j\}:\text{ $i,j\in V_p$ lie in the same block and $i\neq j$}\}.\] Any set $\mathcal V_p$ that contains at least one element from every block and does not contain two elements from the same block forms  a 1-dominating set for $G$ with the maximum degree
\[\Delta_p=\max_{1\leqslant k\leqslant q}d_{pk}-1.\]
So, under the conditions of Theorem \ref{t3},  \eqref{t3-1} reduces to
\[\var(\x_p^\top A_p\x_p)\leqslant C_1 p\|A_p\|^2K\max_{1\leqslant k\leqslant q}d_{pk}.\]
This bound should be compared with Theorem 1.8 of \cite{BVH}, stating that for the isotropic case with $\e\x_p\x_p^\top=I_p$,
\[\var(\x_p^\top A_p\x_p)\leqslant \|A_p\|^2\Big(K\sum_{k=1}^qd_{pk}^2+ 2p\Big).\]
In the asymptotic regime $p\to\infty$ and $K$ not depending on $p$, both the bounds will yield (A1) for $\Sigma_p=\e\x_p\x_p^\top$ only if $\max_{k}d_{pk}=o(p)$. As shown in \cite{BVH}, the last condition on $d_{pk}$ is an optimal condition ensuring that the limiting spectral distribution of $\wh\Sigma_n$ is the Marchenko-Pastur law for the block-independent model (in the case $\e\x_p\x_p^\top=I_p$).

Notice also that by the Gnedenko-Kolmogorov conditions for relative stability (e.g., see (A) and (B) in \cite{H}), the condition $\x_p^\top\x_p/p\pto c$ in Theorem \ref{mdep} could be equivalently replaced by the following Lindeberg condition: for all $\varepsilon>0$,
\begin{equation}\label{lin}\sum_{k=1}^q \e Z_{pk} I(Z_{pk}>\varepsilon)\to0,\quad p\to\infty, \end{equation}
where $Z_{pk}=p^{-1}\sum_{i\in I_{pk}} X_{pi}^2 $.
As a result, Theorem \ref{mdep} allows to compute the limiting spectral distribution of $\wh\Sigma_n$ under the Lindeberg condition, which is much weaker than the fourth moment condition of Theorem 1.3 in \cite{BVH}.
\end{example}

\begin{example}\label{mdep1}\normalfont ($m$-dependent model) Let $m\in[\![0,p-1]\!]$ and suppose the entries of $\x_p=(X_{pi})_{i=1}^p$ are $m$-dependent, i.e. $\{X_{pi}\}_{i\leqslant k} $ and $\{X_{pj}\}_{j>k+m}$ are independent for all $k\in\bN$ with $k+m\leqslant p$. Assume also that  the entries have zero mean and fourth moment bounded by $K>0$. Then $\x_p$ follows the graph dependent model with $G_p=(V_p,E_p)$, where $V_p=[\![1,p]\!]$ and $E_p=\{\{i,i+k\}:\text{ $i,i+k\in V_p,$ $1\leqslant k\leqslant m$}\}.$ 
The maximum degree of $G_p$ is $2m$ and one can construct a 5-dominating set of $G_p$ as $\mathcal V_p=\{k(m+1)\in V_p: k\in[\![1,p/(m+1)]\!]\}$.
So, under the conditions of Theorem \ref{t3},  \eqref{t3-1} reduces to
\[\var(\x_p^\top A_p\x_p)\leqslant  C_5 (2m+1)pK\|A_p\|^2.\]
In the asymptotic regime $p\to\infty$  with $m=m(p)$ and $K$ not depending on $p$, the last inequality guarantees that (A1) holds only if $m=o(p)$. The last condition on $m$ is tight. This follows from Example \ref{bi} and the fact that the block-independent model could be considered as a particular case of $m$-dependent model, as one can always rearrange the entries of $\x_p$ in the block-independent model from Example \ref{bi} in a way that $\x_p$ satisfies the $m$-dependent model with $m=\max_k d_k-1$. 

By the Gnedenko-Kolmogorov conditions for relative stability, the condition $\x_p^\top\x_p/p\pto c$ in Theorem \ref{mdep} could be replaced by the Lindeberg condition \eqref{lin} with  \[Z_{pk}=\frac1p\sum_{i=(k-1)(m+1)+1}^{ k(m+1)} X_{pi}^2,\quad k\in[\![1,p/(m+1)]\!].\]
This could be verified by writing the sum $\x_p^\top\x_p=\sum_{i=1}^pX_{pi}^2$ as two sums with independent entries 
\[\sum_{k}  Z_{pk}\,\bI(k\text{ is odd})+\sum_{k}  Z_{pk}\,\bI(k\text{ is even}).  \]
In fact, one could show that $\x_p^\top\x_p/p\pto c$ is equivalent to the Lindeberg condition using the results of \cite{H} and Lemma \ref{id}. In the case of the $m$-dependent model, the result of Theorem \ref{mdep} improves the results of \cite{HP} (where fixed $m$ is considered) and \cite{WYY} (where $m=o(p^{1/4})$).

Notice that the above results could be easily extended to multidimensional versions of the $m$-dependent model where the entries of $\x_p=(X_i)_{i\in V_p}$ are indexed by a set $V_p\subset \bR^q$ with $|V_p|=p$ and are such that $\{X_i\}_{i\in I}$ and $\{X_j\}_{j\in J}$ are independent when some given distance in $\bR^q$ between $I,J\subseteq V_p$ is greater than $m$.
\end{example}

\section{Proofs}
\label{proofs}
\begin{proof}[Proof of Theorem \ref{t3}.]
First, we will prove the second bound of the Theorem. Let symmetric $A=(a_{ij} )_{i,j=1}^p$ be such that $a_{ij}=0$ if $d_G(i,j)>2$. We have 
 \[\var(\x^\top   A\, \x)=\sum_{1\leqslant i,j,k,l\leqslant p}a_{ij}a_{kl}\cov(X_iX_j,X_kX_l) \bI(d_G(i,j)>2,d_G(k,l)>2)\]
Introduce a zero-mean Gaussian vector $\z=(Z_k)_{k=1}^p$ in $\bR^p$ with $\e \z\z^\top=\Sigma=\e \x\x^\top$. We will show that
\begin{align}\label{z}
&\text{if $i,j,k,l\in V$ and $d_G(i,j)>2,$ $d_G(k,l)>2$,}\nonumber\\ 
&\text{then }\cov(X_iX_j,X_kX_l)=\cov(Z_iZ_j,Z_kZ_l).
\end{align}
Suppose for a moment that this is true. By Lemma 2.3 in \cite{Mag}, 
\[\var(\x^\top   A\, \x)=\var(\z^\top   A\, \z) =2\tr((A\Sigma)^2).\]
By the Cauchy-Schwartz inequality for traces,
\[\tr(A\Sigma A\Sigma)\leqslant \sqrt{\tr((A\Sigma A)^2)\tr(\Sigma^2)}.\]
Furthermore, $\|A\|^2I_p-A^2$ is a positive semidefinite matrix. Therefore,
\[\tr((A\Sigma A)^2)=\tr(A\Sigma A^2\Sigma A)\leqslant \|A\|^2\tr(A\Sigma^2A)=\|A\|^2\tr(\Sigma A^2\Sigma)\leqslant \|A\|^4\tr(\Sigma ^2) \]
and we get the desired bound $\var(\x^\top   A\, \x)\leqslant 2 \|A\|^2 \tr(\Sigma ^2)$.

So, it remains to verify \eqref{z}. We claim that either the sets $\{i,k\}$ and $\{j,l\}$  or the sets   $\{i,l\}$ and $\{j,k\}$ are non-adjacent. If $\{i,j\}$ and $\{k,l\}$ are non-adjacent, then the claim is obvious. Suppose $\{i,j\}$ and $\{k,l\}$ are adjacent. W.l.o.g. $d_G(i,k)\leqslant 1$. We have $d_G(i,j)>2$, $d_G(k,l)>2,$ $d_G(i,l)>1$, and $d_G(k,j)>1$, where the last two inequalities follow from 
\begin{align*}
2&<d_G(k,l)\leqslant d_G(k,i)+d_G(i,l)\leqslant 1+d_G(i,l),\\
2&<d_G(i,j)\leqslant d_G(i,k)+d_G(k,j)\leqslant 1+d_G(k,j).
\end{align*}So, $\{j,l\}$ and $\{i,k\}$ are  non-adjacent. The claim is verified. 

Consider the case when $\{i,k\}$ and $\{j,l\}$ are non-adjacent. The definition of the graph dependent model gives $\e X_iX_jX_kX_l=\e X_iX_k\e X_jX_l$ and $\e X_qX_r=\e Z_qZ_r=0$ for all $q\in\{i,k\},$ $r\in\{j,l\}$. As a result, $(Z_i,Z_k)$ and $(Z_j,Z_l)$ are independent and \[
\cov(X_iX_j,X_kX_l)=\e X_iX_j X_kX_l=\e X_iX_k\e X_jX_l=\]\[=\e Z_iZ_k\e Z_jZ_l=  \e Z_iZ_kZ_jZ_l=\cov(Z_iZ_j,Z_kZ_l).\] Likewise, we get \eqref{z} in the case when $\{i,l\}$ and $\{j,k\}$ are non-adjacent.

Let us prove the first bound of the Theorem.  Let $\mathring A =(\mathring a_{ij} )_{i,j=1}^p $ be defined by 
\begin{equation}\label{aring}
\text{$\mathring a_{ij} =a_{ij}$ if there exists $v\in\mathcal V$ such that $i,j \in  B_2(v)$ and $\mathring a_{ij}=0$ otherwise,}
\end{equation}
where $\mathcal V$ is a $d$-dominating set of $G$.  By the Cauchy inequality, 
 \begin{equation}\label{c}
  \var(\x^\top   A\, \x)\leqslant 2 \var(\x^\top \mathring A \,\x)+2\var(\x^\top   (A-\mathring A)\, \x).
   \end{equation}

For any  $i\in V$, we can always find $v_i\in \mathcal V$ with $d_G(i,v_i)\leqslant 1$. In view of the definition of a $d$-dominating set $\mathcal V$, the latter implies that for all $(i,j)\in V$,
\begin{equation}\label{eq-d}
\sum_{v\in \mathcal{V}}\bI((i,j)\in B_2^2(v))\leqslant\sum_{v\in \mathcal{V}}\bI(i\in B_2(v))\leqslant \sum_{v\in \mathcal{V}}\bI(v_i\in B_3(v))\leqslant d  .
\end{equation}
Let further $\nu_s$ be a subset of $\mathcal V$ with $|\nu_s|=s$ ($\leqslant p$) and set \[B_{\nu_s}:=\bigcap_{v\in\nu_s}B_2(v).\] By the standard properties of the Cartesian product  with respect to intersections,
\[B_{\nu_s}^2=\bigcap_{v\in\nu_s}B_2^2(v) .\] Also, by \eqref{eq-d}, $B_{\nu_s}^2=\varnothing$ when $s>d$ and
\[\bI\Big(\bigcup_{v\in \mathcal V}  B_2^2(v)\Big)=1-\prod_{v\in \mathcal V}(1-\bI ( B_2^2(v)))=\sum_{s=1}^d (-1)^{s-1}\sum_{\nu_s}\bI( B^2_{\nu_s}),\]
where the last sum is taken over all possible $\nu_s$. In particular, we infer that 
\begin{align}\label{pois}
\x^\top\mathring A \,\x&=\sum_{i,j=1}^pa_{ij}X_iX_j\bI\Big((i,j)\in \bigcup_{v\in \mathcal V}  B_2^2(v)\Big)\\
&=
\sum_{s=1}^d (-1)^{s-1}\sum_{\nu_s }\sum_{i,j=1}^pa_{ij}X_iX_j\bI((i,j)\in B^2_{\nu_s})
&\\&
=
\sum_{s=1}^d (-1)^{s-1}\sum_{\nu_s}\x_{\nu_s}^{\top}A_{\nu_s} \x_{\nu_s},
\end{align}
hereinafter $\x_{\nu_s}=(X_{i}:i\in B_{\nu_s})$ and $A_{\nu_s}=(a_{ij}:i,j\in B_{\nu_s})$ if $B_{\nu_s}\not =\varnothing$ and $\x_{\nu_s}=A_{\nu_s}=0$ otherwise.
As $A_{\nu_s} $ is a principal submatrix of $A$, we have that  $\|A_{\nu_s}\|\leqslant \|A\|$  and 
\[
| \x^\top \mathring A \,\x|\leqslant  
\|A\|\sum_{s=1}^d  \sum_{\nu_s}\x_{\nu_s}^{\top}  \x_{\nu_s}=\|A\|\sum_{s=1}^d  \sum_{\nu_s}\sum_{i\in B_{\nu_s}} X_i^2 =\|A\|\sum_{s=1}^d  \sum_{i=1}^p X_i^2\sum_{ \nu_s} \bI(i\in  B_{\nu_s})
\] 
Since any $i\in V$ belongs to no more than $d$ sets  $B_{2}(v)$ with $v\in\mathcal V$ (see \eqref{eq-d}), there are no more than ${d}\choose{s}$ sets of the form $B_{\nu_s}=\bigcap_{v\in\nu_s}B_2(v) $ that cover $i$, i.e. 
\begin{equation}\label{cov}
\sum_{ \nu_s} \bI(i\in  B_{\nu_s})\leqslant {{d}\choose{s}}.
\end{equation}
Therefore,
\begin{equation}\label{norm}
|\x^\top \mathring A \,\x|\leqslant  \|A\|
\sum_{s=1}^d   \sum_{i=1}^p X_i^2 {{d}\choose{s}}\leqslant  (2^d-1)\|A\| \x^\top \x
\end{equation}
In fact, the last inequalities hold for any nonrandom $\x\in \bR^p,$ as follows from the proof. Hence, $\|A-\mathring A\|\leqslant \|A\|+\|\mathring A\|\leqslant 2^d\|A\|$. Putting $(b_{ij})_{i,j=1}^p:=A-\mathring A$, let us show that $b_{ij}=0$ when $d_G(i,j)\leqslant 2$. First, suppose $d_G(i,j)\leqslant 1$. We can always find $v_i\in \mathcal V$ such that $d_G(v_i,i)\leqslant 1$. This shows that \[d_G(v_i,j)\leqslant d_G(v_i,i)+d_G(i,j)\leqslant 2 ,\] i.e. $j\in B_2(v_i)$, and, by the definition of $\mathring A$, $b_{ij}=0$. Suppose that $d_G(i,j)=2$. Then there is $k\in V$ such that $d_G(i,k)=d_G(k,j)=1$. As we have just shown, the latter implies that $i,j\in B_2(v_k)$ and $b_{ij}=0$.

Applying the second bound of the Theorem yields
\begin{equation}\label{dif}
\var(\x^\top   (A-\mathring A)\, \x)\leqslant 2\|A-\mathring A\|^2 \tr(\Sigma^2)\leqslant 2^{2d+1}\|A\|^2 \tr(\Sigma^2).
\end{equation}
Noticing that $\e X_iX_j=0$ when $d_{G}(i,j)>1$, we conclude that \[\tr(\Sigma^2)=\sum_{i=1}^p (\e X_i^2)^2+\sum_{\substack{1\leqslant i,j\leqslant p\\ d_G(i,j)=1 }}2(\e X_iX_j)^2\leqslant  \sum_{i=1}^p (\e X_i^2)^2+\]
\begin{equation}\label{sec}
+\sum_{\substack{1\leqslant i,j\leqslant p\\ d_G(i,j)=1 }}\big((\e X_i^2)^2+(\e X_j^2)^2\big)= 
 \sum_{i=1}^p (\e X_i^2)^2|B_1(i)|\leqslant (\Delta+1) \sum_{i=1}^p  \e X_i^4.
\end{equation}

Also, by the Cauchy–Bunyakovsky–Schwarz inequality and \eqref{pois},
\[\var(\x^\top \mathring A\, \x)\leqslant d\sum_{s=1}^d \var\Big(\sum_{\nu_s }\x_{\nu_s}^{\top}A_{\nu_s} \x_{\nu_s}\Big).\]
For an arbitrary term in the last sum, we have 
\[\var\Big(\sum_{\nu_s }\x_{\nu_s}^{\top}A_{\nu_s} \x_{\nu_s}\Big)=\sum_{  \nu_s,\kappa_s}\cov(\x_{\nu_s}^{\top}A_{\nu_s} \x_{\nu_s},\x_{\kappa_s}^{\top}A_{\kappa_s} \x_{\kappa_s}),\]
where the sum is taken over all $\nu_s,\kappa_s\subseteq \mathcal V$ with $|\nu_s|=|\kappa_s|=s$.
Notice that since $A_{\nu_s} $ is a principal submatrix of $A$, we have that  $\|A_{\nu_s}\|\leqslant \|A\|$ and
\begin{align*}
 \var(\x_{\nu_s}^{\top}A_{\nu_s} \x_{\nu_s})\leqslant  \e |\x_{\nu_s}^{\top}A_{\nu_s} \x_{\nu_s} |^2\leqslant \|A_{\nu_s}\|^2   \e |\x_{\nu_s}^{\top} \x_{\nu_s} |^2\leqslant  \|A\|^2 |B_{\nu_s}| \sum_{i\in B_{\nu_s}}\e X_i^4.
\end{align*}
This and the Cauchy–Bunyakovsky–Schwarz inequality imply that 
\[
\cov(\x_{\nu_s}^{\top}A_{\nu_s} \x_{\nu_s},\x_{\kappa_s}^{\top}A_{\kappa_s} \x_{\kappa_s})\leqslant  \frac12(\var(\x_{\nu_s}^{\top}A_{\nu_s} \x_{\nu_s})+\var(\x_{\kappa_s}^{\top}A_{\kappa_s} \x_{\kappa_s}))\leqslant\]\[\leqslant  \frac{\|A\|^2}2 \Big(|B_{\nu_s}| \sum_{i\in  B_{\nu_s}}\e X_i^4 +|B_{\kappa_s}| \sum_{i\in B_{\kappa_s}}\e X_i^4\Big).
\]
Noting that $\x_{\nu_s}$ and $\x_{\kappa_s}$ are independent when $B_{\nu_s}, B_{\kappa_s}$ are non-adjacent and setting $\bI(\nu_s , \kappa_s):=\bI(\text{$B_{\nu_s}, B_{\kappa_s}$ are adjacent and nonempty})$, we get that
\[
\var\Big(\sum_{\nu_s }\x_{\nu_s}^{\top}A_{\nu_s} \x_{\nu_s}\Big)\leqslant  \frac{ \|A\|^2}2 \sum_{  \nu_s,\kappa_s}\Big(|B_{\nu_s}| \sum_{i\in B_{\nu_s}}\e X_i^4 +|B_{\kappa_s}| \sum_{i\in B_{\kappa_s}}\e X_i^4\Big)\bI(\nu_s , \kappa_s)=\]
\begin{equation}\label{star}
= \|A\|^2 \sum_{  \nu_s,\kappa_s} |B_{\nu_s}| \bI(\nu_s , \kappa_s) \sum_{i\in B_{\nu_s}}\e X_i^4 = \|A\|^2 \sum_{  \nu_s} |B_{\nu_s}|  \sum_{i\in B_{\nu_s}}\e X_i^4 \sum_{\kappa_s}\bI(\nu_s , \kappa_s).
\end{equation} 

Let us estimate the last sum (over $\kappa_s$), which is simply the number of $B_{\kappa_s}$ such that $B_{\nu_s}, B_{\kappa_s}$ are adjacent. Fix $s\in[\![1,d]\!]$. By the definition of $\nu_s$ and $\kappa_s$, if  $B_{\nu_s}, B_{\kappa_s}$ are adjacent, then there are $v\in \nu_s ,$ $u\in \kappa_s$, $k,l\in V$ such that $d_G(v,k)\leqslant 2$, $d_G(u,l)\leqslant 2$, and $d_G(k,l)\leqslant 1$. Then, denoting as above by $v_i\in \mathcal V$ any vertex adjacent to $i\in V$ and using the triangle inequality, we see that $d_G(v,v_k)\leqslant 3$,  $d_G(v_k,v_l)\leqslant 3$, and $d_G(v_l,u)\leqslant 3$. By the definition of $d=d(\mathcal V)$, the number of such sequences $(v,v_k,v_l,u,\kappa_s)$ does not exceed $sd^3 {{d}\choose{s}}$ when  $\nu_s$ is fixed. Indeed, replacing $v_k,v_l$ by arbitrary $w,r\in\mathcal V$, we may count such sequences as follows.

Given $\nu_s,$ one can choose $v\in \nu_s$ in $s$ ways. 

Given $v$, one can choose $w\in \mathcal V$ with $d_G(v,w)\leqslant 3 $ in no more than $d$ ways.

Given $w$, one can choose $r\in \mathcal V$  with $d_G(w,r)\leqslant 3 $ in no more than $d$ ways.

Given $r$, one can choose $u\in \mathcal V$  with $d_G(r,u)\leqslant 3 $  in no more than $d$ ways.

Given $u$, one can choose $\kappa_s\subseteq \mathcal V$  with $|\kappa_s|=s $, $u\in B_{\nu_s}$, and $B_{\kappa_s}\not= \varnothing$  in no more than ${d}\choose{s}$ ways. This follows from   
$\sum_{ \kappa_s} \bI(u\in  B_{\kappa_s})\leqslant {{d}\choose{s}}
$
(see \eqref{cov}).

This proves that \[\sum_{\nu_s}I(\nu_s,\kappa_s)\leqslant sd^3 {{d}\choose{s}}\leqslant d^42^{d}\]
and, by \eqref{cov}, 
\[
\var\Big(\sum_{\nu_s }\x_{\nu_s}^{\top}A_{\nu_s} \x_{\nu_s}\Big)\leqslant   d^42^{d}  \|A\|^2 \sum_{  \nu_s} |B_{\nu_s}|  \sum_{i=1}^p\e X_i^4 \bI(i\in B_{\nu_s}) \leqslant \]\[\leqslant
d^42^{d}  \|A\|^2 m   \sum_{i=1}^p\e X_i^4 \sum_{  \nu_s}\bI(i\in B_{\nu_s}) \leqslant d^42^{d}  \|A\|^2 m    {{d}\choose{s}}\sum_{i=1}^p\e X_i^4 
,\]
where $m=\max\{|B_2(v)|:v\in\mathcal V\}.$ Let us bound $m$ from above. Fix $v\in\mathcal V$. If $i\in B_2(v)$, then $d_G(i,v)\leqslant 2$ and $d_G(v_i,v)\leqslant 3$, where as above, $v_i\in\mathcal V$ is such that $d_G(i,v_i)\leqslant 1$. The set $B_1(v_i)$ contains no more than $\Delta+1$ vertices and there are no more than $d$ different $v_i\in\mathcal V$ with   
 $d_G(v_i,v)\leqslant 3$. Therefore, $|B_2(v)|\leqslant d(\Delta+1)$.

Combining the above bounds yields
\[\var(\x^\top \mathring A\, \x)\leqslant d^7 2^{d}\|A\|^2 (\Delta+1) \sum_{s=1}^d  {{d}\choose{s}} \sum_{i=1}^p\e X_i^4\leqslant   d^7 2^{2d}\|A\|^2 (\Delta+1) \sum_{i=1}^p\e X_i^4.\]
The latter,  \eqref{c}, and \eqref{sec} prove the first inequality of the Theorem.
\end{proof}

\begin{proof}[Proof of Theorem \ref{mdep}.]
The desired results will follow from Theorem \ref{rm} and  \ref{mp1} if we verify (A1) and (A2) for $\Sigma_p=\e \x_p\x_p^\top$. First, note that (A2) follows from  \eqref{sec}, $\Delta_p=o(p)$, and the fact that the entries of $\x_p=(X_{ip})_{i=1}^p$ have uniformly bounded second moments. 

Let us verify (A1). Let $\mathcal{V}_p$ be a $d_p$-dominating set of $G_p$ for any $p=p(n)$. We will use the notation and constructions from the proof of Theorem \ref{gd}.
For $p=p(n)$, consider arbitrary symmetric $A_p=(a_{ijp})_{i,j=1}^p\in\bR^{p\times p}$ with $\|A_p\|\leqslant 1$ and define $\mathring A_p=(\mathring a_{ijp})_{i,j=1}^p$ as in \eqref{aring}, i.e. \begin{center}
$\mathring a_{ijp}=a_{ijp}$ if $i,j\in B_{2p}(v):=\{k\in[\![1,p]\!]:d_{G_p}(k,v)\leqslant 2\} $ for some $v\in \mathcal V_p$. 
\end{center}

It is shown above \eqref{sec} that the $(i,j)$th element of $A_p-\mathring A_p$ is zero when the distance between $i$ and $j$ in $G_p$ does not exceed 2. This implies that 
\[\e \x_p^\top(A_p-\mathring A_p)\x_p=\sum_{  i,j=1}^p\Big(a_{ijp}-\mathring a_{ijp}\Big)\e X_{ip}X_{jp} \bI( d_{G_p}(i,j)>2)=0,\]
\[\e \x_p^\top \mathring A_p \x_p=\e \x_p^\top A_p \x_p=\e\,\tr( \x_p\x_p^\top A_p )=\tr(\e \x_p\x_p^\top A_p)=\tr(\Sigma_p A_p).\] Furthermore, as $d_p=O(1)$, \eqref{dif} implies that 
\[\var(\x_p^\top   (A_p-\mathring A_p) \x_p)=\e|\x_p^\top   (A_p-\mathring A_p) \x_p|^2\leqslant O(1) \tr(\Sigma_p^2) \]
and, as a result, by (A2), 
 \[\frac{\x_p^\top   (A_p-\mathring A_p) \x_p}p\to 0.\]
 
To finish the proof of (A1), we need to show that  
 \[\frac{\x_p^\top    \mathring A_p \x_p-\e \x_p^\top    \mathring A_p \x_p}p\to 0.\]
As in the proof of Theorem \ref{t3}, we let further $\nu_s=\nu_s(p) \subseteq \mathcal V_p$ with $|\nu_s|=s$ for $s \in[\![1,d_p ]\!]$ and set 
\begin{center}$B_{\nu_s p}=\bigcap\limits_{v\in \nu_s}  B_{2p}(v) $, $\x_{\nu_s p}=(X_{ip}:i\in B_{\nu_sp})$, and $A_{\nu_sp}=(a_{ijp}:i,j\in B_{\nu_sp})$,\end{center} we can   write
\[\x_{\nu_sp}^\top  A_{\nu_s p}\x_{\nu_sp}= \x_{vp }^\top   \bar A_{\nu_sp }\x_{v  p} ,\] where
$v=\min\{u:u\in\nu_s\},$ $\x_{v  p}=\big(X_{ip}:i\in B_{2p}(v )\big),$ and \[\bar A_{\nu_s p}=(a_{ijp}\bI\big(i,j\in  B_{\nu_sp})):i,j\in B_{2p}(v)\big) .\] By such definitions,  $\|\bar A_{\nu_s p}\|=0$ if $B_{\nu_s p}=\varnothing$  and $\|\bar A_{\nu_s p}\|=\|  A_{\nu_s p}\|\leqslant\|A_p\|\leqslant 1$ otherwise. Therefore, by \eqref{pois}, 
\[\x_p^\top    \mathring A_p \x_p= \sum_{s=1}^{d_p} (-1)^{s-1}\sum_{\nu_s}\x_{\nu_s p}^{\top}A_{\nu_sp} \x_{\nu_s p}=
\sum_{v\in\mathcal V_p}   \x_{vp}^{\top}D_{vp} \x_{vp},\]
where \[ D_{vp}=\sum_{s=1}^{d_p} (-1)^{s-1}\sum_{\nu_s} \bar A_{\nu_sp} \bI(v=\min\{u:u\in\nu_s\}, B_{\nu_sp}\not=\varnothing). \] 
By construction and \eqref{norm},  
\[ |\x_{vp}^{\top}D_{vp} \x_{vp}|=\Big|\sum_{s=1}^{d_p} (-1)^{s-1}\sum_{\nu_s} \x_{\nu_sp}^\top   A_{\nu_sp}\x_{\nu_sp} \bI(v=\min\{u:u\in\nu_s\}, B_{\nu_sp}\not=\varnothing)\Big|\leqslant \]\[\leqslant \sum_{s=1}^{d_p}  \sum_{\nu_s}  |\x_{\nu_sp}^\top   A_{\nu_sp}\x_{\nu_sp} | \leqslant \sum_{s=1}^{d_p}  \sum_{\nu_s}\x_{\nu_s p}^{\top} \x_{\nu_s p}\leqslant 2^{d_p} \x_p^\top \x_p.\]
In fact, the last inequality holds for any nonrandom $\x_p\in\bR^p$ (see the proof of \eqref{norm}). In particular, this shows that $\|D_{vp}\|\leqslant 2^{d_p}$ for all $v\in\mathcal V_p.$

Consider a new graph $\Gamma_p$ with the vertex set $\mathcal  V_p$ and the edge set \[\{\{u,v\}: B_{2p}(u),B_{2p}(v)\text{ are adjacent}\}.\]
The maximum degree of $\Gamma_p$ does not exceed $d_p^3$, as shown in the proof of Theorem \ref{t3} after \eqref{star}.
Therefore, its vertices can be partitioned into no more than $d_p^3+1$ sets $I_{ kp}$ ($1\leqslant k\leqslant d_p^3+1$) in a way that in each set, no two vertices are adjacent (this follows from so-called greedy coloring of the graph $\Gamma_p$). This means that for each $k,$ the vectors $\x_{vp}$ are mutually independent over $v\in I_{kp}$.
Note  that $|\x_{vp}^\top D_{vp}\x_{vp}|\leqslant \|D_{vp}\|\x_{vp}^\top \x_{vp}\leqslant 2^{d_p}  \x_{vp}^\top \x_{vp}$ and
\[\frac1{2^{d_p}p}\Big|\sum_{v\in I_{kp}}\x_{vp}^\top D_{vp}\x_{vp}\Big|\leqslant \frac{1}p \sum_{v\in I_{kp}}\x_{vp}^\top  \x_{vp}\leqslant \frac{ \x_{p}^\top \x_{ p}}{p}\pto c.\]
We need the following lemma.  

\begin{lemma}\label{id} Let $\{X_{ nk}: 1\leqslant k\leqslant k_n ,n\in \bN\}$ be a row-wise independent triangular array of random variables with finite means and such that for each $\varepsilon>0,$
\begin{equation}\label{asynegl}
\lim_{n\to\infty} \max_{1\leqslant k\leqslant k_n}\p( |X_{nk}|>\varepsilon)=0.
\end{equation}
Suppose also $\{T_n\}_{ n\geqslant 1} $ are nonnegative random variables satisfying for all $n\in\bN,$   $\e T_n<\infty$ and $|S_n|\leqslant T_n$ a.s., where $S_n=\sum_{k=1}^{k_n} X_{nk}$. If there exists $c>0$ such that $T_n\pto c$ and $\e T_n\to c$ as $n\to\infty$, then
$S_n-\e S_n\pto0.$
\end{lemma}
The proof of the Lemma can be found in the Appendix. By the assumptions of the Theorem, 
\[ \frac{ \x_{p}^\top \x_{ p}}{p}\pto c\quad \text{and}\quad  \frac{ \e\x_{p}^\top \x_{ p}}{p}=\frac{\tr(\Sigma_p)}{p}\pto c.\]
Also, \[\p( |\x_{vp}^\top D_{vp}\x_{vp}|>\varepsilon  2^{d_p}p)\leqslant \frac{\e |\x_{vp}^\top D_{vp}\x_{vp}|}{ \varepsilon  2^{d_p}p}\leqslant \]\[\leqslant \frac{\e  \x_{vp}^\top  \x_{vp} }{ \varepsilon   p}\leqslant \frac{K|B_2(v)|}{\varepsilon p}\leqslant \frac{d_p(\Delta_p+1)K}{\varepsilon p}\to 0 \]
uniformly in $v\in\mathcal V_p$ for any $\varepsilon\to 0$, where the bound $|B_2(v)|\leqslant d_p(\Delta_p+1)$ was established at the end of the proof of Theorem \ref{t3}. 
Hence, applying Lemma \ref{id} yields
\[\frac1{2^{d_p}p}\sum_{v\in I_{kp}}(\x_{vp}^\top D_{vp}\x_{vp}-\e \x_{vp}^\top D_{vp}\x_{vp})\pto 0\] for all sequences 
$k=k(p)$ with values $1\leqslant k(p)\leqslant d_p^3+1$. As $d_p=O(1)$, the latter implies that
\[\frac{\x_p^\top    \mathring A_p \x_p-\e \x_p^\top    \mathring A_p \x_p}p= \frac1{p}\sum_{1\leqslant k\leqslant d_p^3+1} \sum_{v\in I_{kp}}(\x_{vp}^\top D_{vp}\x_{vp}-\e \x_{vp}^\top D_{vp}\x_{vp})\pto 0.\]
The proof of the Theorem is finished.
\end{proof}

\section*{Appendix}
\begin{proof}[Proof of Lemma \ref{id}]
We will prove the desired result by contradiction. Suppose $S_n-\e S_n $ does not tend to zero in probability.  Therefore, one can find $\varepsilon,\delta>0$ and an infinite set  $\mathcal N\subseteq \bN$ such that 
\begin{equation}\label{contr}
\text{$\p(|S_n-\e S_n|> \varepsilon)\geqslant \delta$ for all $n\in\mathcal N$.}
\end{equation}

Taking $K>0$ large enough we can make $\p(|S_n|>K)$ arbitrarily small, as
\[\p(|S_n|>K)\leqslant\frac{\e |S_n|}K\leqslant \frac{\e T_n}K \to \frac{c}K.\]
With this, one can easily deduce that $\{\mu_n\}_{n\in \mathcal N}$ is a tight collection of probability measures, where $\mu_n$ is  the distribution of $S_n$. By Prokhorov's theorem, there exists an increasing sequence $(n_l)_{l=1}^\infty$ with $n_l\in \mathcal N$ and a random variable $\sigma$ such that $S_{n_l}$ converges in distribution to $\sigma$. 
In view of \eqref{asynegl}, the Khintchin theorem implies that  $\sigma$  has an infinitely divisible distribution (see Theorem 3.1 in \cite{Pe}). As $|S_n| \leqslant T_n\pto c$, $\sigma$ should be bounded a.s. This is possible only if $\sigma$ is constant a.s. (see Corollary 3 in \cite{BS}). Therefore, $S_{n_l}$ also converges to $\sigma$ in probability.

Suppose for a moment that the collection $\{T_{n }\}_{n=1}^\infty$ is  uniformly integrable. Then the collection $\{S_{n }\}_{n=1}^\infty$ is uniformly integrable and, by the Lebesgue-Vitali theorem (Theorem 4.5.4 in \cite{Bo}), $\e| S_{n_l}-\sigma|\to 0$ and, in particular, $\e  S_{n_l}- \sigma\to 0$ as $l\to\infty.$ The latter would contradict  to the assumption \eqref{contr}. This would prove that \eqref{contr} could not hold and $ S_n-\e S_n\pto 0$.

To finish the proof, we need to prove the uniform integrability of $T_{n }$. This follows from $0\leqslant T_n \to c$  and $\e T_n\to c.$ Indeed, by the Lebesgue dominated convergence theorem,
$\e T_n\bI(T_n\leqslant 2c)\to c\bI(c\leqslant 2c)=c$. Hence,  \[\lim_{n\to\infty}\sup_{k\geqslant n} \e T_k\bI(T_k>2c)= 0.\]
By the absolute continuity of the Lebesgue integral, we can find $(C_n)_{n=1}^\infty$ such that  $2c\leqslant C_n\to \infty$ and 
\[\sup_{1\leqslant k\leqslant n} \e T_k\bI(T_k>C_n)\to 0.\]
As a result, we see that $\e T_k\bI(T_k>C_n)\to0$ as $n\to\infty$ uniformly in $k$. This means that the sequence $(T_n)_{n=1}^\infty$ is uniformly integrable. The proof of the Lemma is finished.
\end{proof}

\end{document}